\documentclass[11pt,reqno]{amsart}
\usepackage{latexsym, amsfonts, amsmath, amsthm, amssymb,amscd,color}
\usepackage{pstricks}
\usepackage{pst-optic}
\allowdisplaybreaks

\newtheorem{thm}{Theorem}[section]

\newtheorem{prop}[thm]{Proposition}

\newtheorem{cor}[thm]{Corollary}

\newtheorem{conj}[thm]{Conjecture}

\numberwithin{equation}{section}

\def\OP{{\mathcal O}{\mathcal P}}
\def\op{\mathrm{op}}

\def\Z{{\mathbb Z}}

\def\Sy{{\mathfrak S}}

\def\P{{\mathbb{P}}}
\def\Z{{\mathbb{Z}}}

\begin{document}

\title[Pattern avoidance in ordered partitions]{Pattern avoidance
in ordered set partitions and words}

\author[A. Kasraoui]{Anisse Kasraoui}
\address{Fakult\"at f\"ur Mathematik, Universit\"at Wien,
Nordbergstrasse 15, A-1090 Vienna, Austria}
\email{anisse.kasraoui@univie.ac.at}
\thanks{Research supported by the grant S9607-N13 from Austrian Science Foundation FWF
 in the framework of the National Research Network  ``Analytic Combinatorics and Probabilistic Number theory".}

\begin{abstract}
We consider the enumeration of ordered set partitions avoiding a
permutation pattern, as introduced by Godbole, Goyt, Herdan and
Pudwell. Let $\op_{n,k}(p)$ be the number of ordered set partitions
of $\{1,2,\ldots,n\}$ into $k$ blocks that avoid a permutation
pattern $p$. We establish an explicit identity between the number
$\op_{n,k}(p)$ and the numbers of words avoiding the inverse of $p$.
This identity allows us to easily translate results on
pattern-avoiding words obtained in earlier works into equivalent
results on pattern-avoiding ordered set partitions. In particular,
\emph{(a)} we determine the asymptotic growth rate of the sequence
$(\op_{n,k}(p))_{n\geq 1}$ for
 every positive $k$ and every permutation pattern $p$,
\emph{(b)} we partially confirm a conjecture of Godbole et al.
concerning the variation of
 the sequences $(\op_{n,k}(p))_{1\leq k\leq n}$,
\emph{(c)} we undertake a detailed study of the number of ordered
set partitions avoiding a pattern of length 3.
\end{abstract}
\maketitle

\section{Introduction}

\subsection{} This paper is concerned with the enumeration of pattern-avoiding ordered set partitions.
The subject, which can be seen as a far-reaching generalization of
the study of pattern-avoiding permutations, was initiated by
Godbole, Goyt, Herdan and Pudwell in~\cite{Godbole}.
 Recall that an ordered partition of a set $S$ is a sequence of nonempty and mutually
disjoint subsets, called blocks, whose union is $S$. When $S$ is a
subset of integers, it is usual to separate the blocks by a slash
and to arrange the elements within each block in increasing order.
We let $\OP_{n,k}$ denote the set of ordered partitions of
$[n]:=\{1,2,\ldots,n\}$ into $k$ blocks. For instance, $2\,7/3/1\,
4\, 8/5\, 6$ and $3/2\,7/1\, 4\, 8/5\, 6$ are two (distinct) members
of $\OP_{8,4}$.

Godbole et al.~\cite{Godbole} considered the following notion of
pattern containment: an ordered partition $\pi=B_1/B_2/\cdots/B_k$
in $\OP_{n,k}$ is said to contain a permutation $p=p_1\,p_2\,\cdots
\,p_m$ in $\Sy_m$, the symmetric group of the set~$[m]$, as a
pattern if there exist integers $i_1$, $i_2$,\ldots, $i_m$, with
$1\leq i_1<i_2<\cdots<i_m\leq k$, and there exists $b_j\in B_{i_j}$
such that $b_1\,b_2\,\cdots \,b_m$ is order isomorphic to~$p$.
Otherwise we say that $\pi$ avoids $p$. For example, $2\,7/3/1\, 4\,
8/5\, 6$ contains the pattern $p= 213$, as evidenced (for instance)
by $b_1=2\in B_1$, $b_2=1\in B_3$ and $b_3=5\in B_4$. The number of
members of $\OP_{n,k}$ that avoid the pattern $p$ will be denoted by
$\op_{n,k}(p)$. The special case $k=n$ is of particular interest
since we clearly have
\begin{align}\label{eq:op-perm}
 \op_{n,n}(p)=s_n(p),
\end{align}
where $s_n(p)$ stands for the number of permutations in $\Sy_n$ that
avoid the pattern~$p$ in the usual sense (see
e.g.~\cite[Chapter~4]{Bona}).

Finding a closed formula for $\op_{n,k}(p)$  is in general a
hopeless task. Obviously, we have $\op_{n,k}(1)=0$ for all $n\geq
k\geq 1$. It is also easy to show~\cite{Godbole} that
\begin{align}
 \op_{n,k}(1\,2)=\op_{n,k}(2\,1)=\binom{n-1}{k-1} \qquad(n\geq k\geq1).
\end{align}
The first nontrivial case is that of patterns of length three.
Godbole et al. proved that $\op_{n,k}(p)$, for any $n,k\geq 1$, is
the same for all $p\in\Sy_3$, i.e. $\op_{n,k}(p)=\op_{n,k}(321)$ for
all $p\in\Sy_3$. By~\eqref{eq:op-perm}, this nicely generalizes
Knuth's remarkable observation (see e.g.~\cite[Chapter~4]{Bona})
that the number of permutations of $[n]$ that avoid the pattern $p$
is the same for all $p\in\Sy_3$. Table~\ref{table:1} lists the first
few values of~$\op_{n,k}(321)$.
\begin{table}[h]\label{table:1}
\centering
\begin{tabular}{|c| r r r r r r r r r r|}
\hline $n\setminus k$
  & 1&  2&    3&    4&     5&     6&     7&     8&      9&   10\\
\hline
1 &1&    &     &     &      &      &      &      &      &     \\
2 &1&   2&     &     &      &      &      &      &      &     \\
3 &1&   6&    5&     &      &      &      &      &      &     \\
4 &1&  14&   27&   14&      &      &      &      &      &     \\
5 &1&  30&   99&  112&    42&      &      &      &      &     \\
6 &1&  62&  307&  564&   450&   132&      &      &      &     \\
7 &1& 126&  867& 2284&  2895&  1782&   429&      &      &     \\
8 &1& 254& 2307& 8124& 14485& 13992&  7007&  1430&      &     \\
9 &1& 510& 5891&26492& 62085& 83446& 65065& 27456&  4862&     \\
10&1&1022&14595&81148&239269&418578&450905&294632&107406&16796\\
\hline
\end{tabular}
\vspace{0.3cm} \caption{The number of 321-avoiding ordered set
partitions of $[n]$ into $k$ blocks}
\end{table}
That the first values of the diagonal $(\op_{n,n}(321))_{n\geq 1}$
coincide with those of the Catalan sequence comes as no surprise
because of Knuth's well-known result (see
e.g.~\cite[Chapter~4]{Bona}) that $s_n(p)$, for all $p\in\Sy_3$, is
the $n$-th Catalan number $C_n=\frac{1}{n+1}\binom{2n}{n}$.

The study of the array $[\op_{n,k}(321)]_{n\geq k\geq1}$, which has
also been considered in~\cite{Godbole,Chen}, is one of the central
theme in this paper with which Section~3 is wholly concerned. We
will also derive interesting results for general pattern $p$. To
this end, we first need to state the key result of this paper which
is a simple identity relating the number of ordered partitions
avoiding a pattern $p$ with the numbers of words avoiding $p^{-1}$,
the inverse of $p$.

\subsection{Pattern-avoiding ordered set partitions and words}  Pattern avoidance has already been studied in
other contexts than permutations, notably in words~\cite{Bur,BrMa}
and (unordered) set partitions~\cite{Kla,Sag}. We refer the reader
to the recent textbooks~\cite{HeMan-w,HeMan-p,Kit} for a survey to
these topics and further references.

There is a well-known intimate connection between ordered set
partitions and words that appears to nicely keep track of pattern
containment. As usual, the set of words $w=w_1\,w_2\,\cdots \,w_n$
of length $n$ with letters in an alphabet $A$ will be denoted by
$A^n$. We also let $SW_{n,A}$ denote the set of words in $A^n$ that
contain at least one occurrence of each $i\in A$. With each member
$\pi=B_1/B_2/\cdots/B_k$ of $\OP_{n,k}$ there is associated a word
$w(\pi)=w_1\,w_2\,\cdots \,w_n$ in $[k]^n$ such that $w_i=j$ if and
only if $i\in B_j$. For instance, $\pi=2\,7/3/1\, 4\, 8/5\, 6$ is
sent to the word $w(\pi)=3\,1\,2\,3\,4\,4\,1\,3$. This
correspondence establishes a bijection between $\OP_{n,k}$ and
$SW_{n,[k]}$ and keeps track of pattern containment as follows: an
ordered partition $\pi$ contains a pattern $p$ if and only if its
associated word $w(\pi)$ contains $p^{-1}$, as is easily verified.
Consequently, we have the relation
\begin{align}\label{eq:OPvsSurjWORDS}
  \op_{n,k}(p)=\left|SW_{n,[k]}(p^{-1})\right| \qquad(n,k\geq 1),
 \end{align}
where $SW_{n,A}(q)$ stands for the set of words in $SW_{n,A}$ that
avoid $q$. Godbole et al.~\cite[Section~3]{Godbole} and Chen et
al.~\cite[Equation~(3.3)]{Chen} also observed this relation, but the
next important identity does not seem to have been observed (though
it is immediate from~\eqref{eq:OPvsSurjWORDS} and the
inclusion-exclusion principle). For $A\subseteq \P$ let $A^n(p)$
denote the set of words in $A^n$ that avoid a pattern~$p$.
\begin{prop}\label{prop:OPvsWORDS}
Let $p$ be a permutation. Then for all $n\geq k\geq 1$, we have
\begin{align}\label{eq:OPvsWO}
\op_{n,k}(p)&=\sum_{j=1}^k \binom{k}{j}(-1)^{k-j} |[j]^n(p^{-1})|.
\end{align}
\end{prop}
\begin{proof}
 From the definition of $SW_{n,[k]}(p^{-1})$ and the inclusion-exclusion principle, we have
 \begin{align}\label{eq:incl-excl}
\left|SW_{n,[k]}(p^{-1})\right|=\sum_{A\subseteq[k]}(-1)^{k-|A|}\left|A^n(p^{-1})\right|.
\end{align}
 Obviously, if $|A|=j$ we have $\left|A^n(p^{-1})\right|=\left|[j]^n(p^{-1})\right|$. If we plug this into~\eqref{eq:incl-excl}
 and then combine the resulting equation with~\eqref{eq:OPvsSurjWORDS}, we arrive at the desired identity.
\end{proof}

Proposition~\ref{prop:OPvsWORDS}, which can be seen as the key
result of this paper, is important for us because it allows us to
easily translate any result on pattern-avoiding words (obtained in
previous works) into an equivalent result on pattern-avoiding
ordered partitions. For example, by~\eqref{eq:OPvsWO}, the result
that $\op_{n,k}(p)$ is the same for any pattern $p\in\Sy_3$ is
equivalent to a result of Burstein~\cite{Bur} which asserts that
$|[k]^n(p)|$ is the same for any pattern $p\in\Sy_3$. We shall not
translate all existing results on pattern-avoiding words
(see~\cite{HeMan-w} for a survey) but we limit our-self to results
that permit us to study questions that Goldbole et al. raised. In
Section~2, we present results on the asymptotic enumeration of
pattern-avoiding ordered partitions. These results are immediately
derived from Proposition~\ref{prop:OPvsWORDS} and results due to
Br\"{a}nd\'en and Mansour~\cite{BrMa}, and Regev~\cite{Reg}. In
Section~3, we shall rely on results of Burstein~\cite{Bur} to
undertake a detailed study of the array~$[\op_{n,k}(321)]_{n\geq
k\geq1}$.

 \section{Growth-rate and the monotonicity}
 Let us denote the length of a pattern $p$ by $|p|$. Clearly, every ordered partition with at most $|p|-1$ blocks
 avoids $p$. So  we have
 \begin{align}\label{eq:SmallPattern}
  \op_{n,k}(p)&=|\OP_{n,k}|=k!\,S(n,k)\qquad(1\leq k<|p|,\;n\geq 1),
 \end{align}
 where $S(n,k)$ are the Stirling numbers of the second kind. From well-known properties of Stirling numbers, it follows
 that the generating function $\sum_{n\geq 0} \op_{n,k}(p)\, x^n$ is
rational (equivalently, the sequence $(\op_{n,k}(p))_{n\geq 1}$
satisfies a linear recurrence) for each positive $k$ less than
$|p|$. It is remarkable that this is actually true, not only for $k$
less than $|p|$, but for every positive $k$.
 \begin{thm}\label{thm:Cor-BraMan}
 For every pattern $p$ and every positive $k$, the generating function $\sum_{n\geq 0} \op_{n,k}(p)\, x^n$  is
rational.
 \end{thm}
This result is immediate from Proposition~\ref{prop:OPvsWORDS} and a
theorem of Br\"{a}nd\'en and Mansour~\cite[Theorem~3.1]{BrMa} which
states that the generating function $\sum_{n\geq 0} |[k]^n(p)|\,
x^n$ is rational for every positive integer~$k$ and every
pattern~$p$. Alternatively, this can also be derived directly by
adapting the proof of a result of Klazar~\cite[Theorem~3.1]{Kla}
stating that the ordinary generating function of (unordered) set
partitions into $k$ blocks avoiding a pattern $p$ is rational for
every positive integer~$k$ and every pattern~$p$.

The preceding theorem, combined with the theory of rational
generating functions, leads to the following theorem which is one of
the main result in~\cite{Godbole}.
  \begin{thm}[Godbole et al.]\label{thm:God}
  For each fixed $k\geq 1$ and every pattern $p$, the limit $\underset{n\to\infty}\lim \left(\op_{n,k}(p)\right)^{\frac{1}{n}}$
  exists in $[0,\infty)$.
 \end{thm}
 It is well-known that  for every positive $k$, $S(n,k)^{\frac{1}{n}}$ tends to $k$ as $n$ tends to infinity.
 This, combined with~\eqref{eq:SmallPattern}, implies that for every
 pattern $p$
 \begin{align}\label{eq:growth-SmallPattern}
 \underset{n\to\infty}\lim \left(\op_{n,k}(p)\right)^{\frac{1}{n}}=k\qquad(1\leq k<|p|).
 \end{align}
From a result of Br\"{a}nd\'en and Mansour, we can also determine
the limit in Theorem~\ref{thm:God} when $k\geq |p|$.

\begin{thm}\label{thm:Growth}
  For every pattern $p$, we have
  \begin{align}\label{eq:growth-pattern}
 \underset{n\to\infty}\lim \left(\op_{n,k}(p)\right)^{\frac{1}{n}}=|p|-1\qquad( k\geq |p|).
 \end{align}
 \end{thm}
 Actually, we can derive a stronger result. Suppose $p\in\Sy_d$ and let $k\geq d$ be given. Then there are a constant
$C>0$ and a nonnegative integer $M$ such that
\begin{align*}
 \op_{n,k}(p) \sim |[k]^n(p^{-1})|\sim C\, n^{M}(d-1)^n\qquad(n\to\infty).
\end{align*}
This result, which easily leads to Theorem~\ref{thm:Growth}, is
immediate from Proposition~\ref{prop:OPvsWORDS} and the next result
of Br\"{a}nd\'en and Mansour~\cite[Theorem~3.2]{BrMa}.

 \begin{thm}[Br\"{a}nd\'en and Mansour]\label{thm:BraMan}\
Suppose $p\in\Sy_d$ and let $k\geq d$ be given. Then there are a
constant $C>0$ and a nonnegative integer $M$ such that
\begin{align}\label{eq:asympt-BrMa}
 |[k]^n(p)|\sim C\, n^{M}(d-1)^n\qquad(n\to\infty).
\end{align}
 \end{thm}

   Godbole et al.~\cite[Section~7]{Godbole} were also interested in the variation of the
 sequences $\left(\op_{n,k}(p)\right)_{k=1..n}$. They noticed that theses sequences are, in general, not monotone in $k$
 but they made the following conjecture that we present here in a slightly different (but equivalent) form.
 \begin{conj}[Godbole et al.]\label{conj:God}
For each pattern $p$ and each integer $k$, there exists a positive
integer $n_0(k,p)$ such that for $n\geq n_0(k,p)$,
$$\op_{n,k+1}(p)>\op_{n,k}(p)>\cdots>\op_{n,|p|}(p).$$
 \end{conj}
  Relying on the next result of Regev~\cite{Reg}, proved even before Theorem~\ref{thm:BraMan} was known, we can confirm this
  conjecture for the increasing patterns $p_{\ell}=1\,2\,\cdots\,(\ell+1)$, $\ell\geq 1$.
\begin{thm}[Regev]\label{thm:Reg}
For $k \geq \ell\geq 1$, we have
\begin{align}\label{eq:Asympt-Reg}
 |[k]^n(p_{\ell})|\sim  \left(\frac{1}{\ell}\right)^{\ell(k-\ell)}
 \left(\prod_{i=0}^{\ell-1}\frac{i!}{(k-1-i)!}\right)\, n^{\ell(k-\ell)}\,\ell^n\qquad(n\to\infty).
\end{align}
 \end{thm}
 Let $k$ and $\ell$ be fixed integer with $k\geq\ell\geq 1$. Then the preceding result, in conjunction with Proposition~\ref{prop:OPvsWORDS},
 shows that $\op_{n,k}(p_{\ell})$ is, as $n$ tends to infinity, equivalent to the right-hand side of~\eqref{eq:Asympt-Reg}.  It is then a routine
 matter to show  that $\op_{n,k}(p_{\ell})/\op_{n,k+1}(p_{\ell})$ tends to 0 as $n$ tends to infinity, which in turn, obviously yields the next result.
\begin{thm}
Conjecture~\ref{conj:God} is true for the patterns
$p_{\ell}=1\,2\,\cdots\,(\ell+1)$, $\ell\geq 1$.
 \end{thm}
 We should also notice that, by Theorem~4.8 in~\cite{BrMa}, we have $|[k]^n(p_{\ell})|=|[k]^n(\widetilde{p_{\ell}})|$ for all $n,k\geq 1$,
 where $\widetilde{p_{\ell}}=1\,2\,\cdots\,(\ell-1)\,\ell\,(\ell+1)$. This, combined with Proposition~\ref{prop:OPvsWORDS} and the
 preceding result, shows that  Conjecture~\ref{conj:God} is also true for the patterns $\widetilde{p_{\ell}}$, $\ell\geq 1$.

\section{The number of ordered set partitions avoiding a pattern of length 3}

Recall that $\op_{n,k}(p)$ is the same for all $p\in\Sy_3$, as
previously noticed. It is convenient to set for all integers $n$ and
$k$
 \begin{align}\label{eq:op321-IniVal}
  \op_{n,k}(321)=\left\{
  \begin{array}{ll}
    0, & \hbox{if $n<k$ or $k<0$ or $n<0$;} \\
    \delta_{n,0}, & \hbox{if $k=0$.}
  \end{array}
\right.
\end{align}
By~\eqref{eq:SmallPattern} and from what we have recalled in the
introduction, we have
 \begin{align}\label{eq:op321-k leq 2}
  \op_{n,1}(321)=1\qquad\text{and}\qquad\op_{n,2}(321)=2^n-2\qquad(n\geq1),
\end{align}
and
 \begin{align}\label{eq:op321-k=n}
  \op_{n,n}(321)=\frac{1}{n+1}\binom{2n}{n}\qquad(n\geq1).
\end{align}
Formulas for $\op_{n,3}(p)$ and $\op_{n,n-1}(p)$  were obtained by
Godbole et al.~\cite[Sections~2 and 3]{Godbole}.
\begin{thm}[Godbole et al.] For all $n\geq1$,
 \begin{align}
  \op_{n,3}(321)&=(n^2+3n-16)2^{n-3}+3,\label{eq:Godbole-k=3}\\
  \op_{n,n-1}(321)&= \frac {3(n-1)^2}{n(n+1)}\binom{2n-2}{n-1}\label{eq:Godbole-conjecture}.
 \end{align}
 \end{thm}
 The nontrivial question of determining the numbers $\op_{n,k}(321)$
for general $n$ and $k$ was raised in~\cite{Godbole}. One of the
main goal of the present section is to establish several properties
of the array $[\op_{n,k}(321)]_{n\geq k\geq1}$ that can be easily
used to extend Table~\ref{table:1}.

In Section~3.1, we recall the generating function of the array
$[|[k]^n(321)|]_{n, k\geq1}$ and the double sum formula for the
numbers $|[k]^n(321)|$, as was obtained by Burstein in~\cite{Bur}.
By simplifying Burstein's formula, we show that $|[k]^n(321)|$ can
actually be written as a single sum. This leads to an explicit
formula for the numbers $\op_{n,k}(321)$ in terms of a double sum.
The formula, which is given in Section~3.2, shows that the second
equation in~\eqref{eq:op321-k leq 2} and~\eqref{eq:Godbole-k=3} are
just instances of a general phenomenon: For each fixed $k\geq2$,
$\op_{n,k}(321)=P_k(n)2^{n}+(-1)^{k-1}k$ for some polynomial
$P_k(n)$ in $n$ of degree $2k-4$.  We also derive there the ordinary
generating function of the array $[\op_{n,k}(321)]_{n\geq k\geq1}$,
which was also obtained by Chen et al.~\cite{Chen}, and we give
several recurrences satisfied by the entries of this array. In
Section~3.3, we indicate another remarkable phenomenon of
which~\eqref{eq:op321-k=n} and~\eqref{eq:Godbole-conjecture} are
just instances: For each fixed $r\geq 0$,
$\op_{n+r,n}(p)=Q_r(n)(2n)!/n!/(n+r+1)!$ for some polynomial
$Q_r(n)$ in $n$ of degree less than or equal to $2r$. In the
appendix at the end of this paper,  we reproduce
 a calculation of Krattenthaler~\cite{Kratt} for proving~\eqref{eq:Godbole-conjecture}
 directly  from our double sum formula.

\subsection{The number of words avoiding a pattern of length 3}

The results of Burstein~\cite[Chapter~3]{Bur} on which we shall rely
in Section~3.2 are collected in the next result.
\begin{thm}[Burstein]\label{thm:Burstein-GFandFor}
For any pattern $p\in\Sy_3$,  every $k\geq2$ and every $n\geq0$,
there holds
\begin{align}\label{eq:Burstein}
|[k]^n(p)|=2^{n-2(k-2)}\sum_{j=0}^{k-2}\sum_{i=j}^{k-2}
\frac{1}{i+1}\binom{2i}{i}\binom{2(k-2-i)}{k-2-i}\binom{n+2j}{2j}.
\end{align}
Moreover, we have
\begin{align}\label{eq:Burstein-GF}
\sum_{n,k\geq 0}|[k]^n(p)|\,x^k
y^n=1+\frac{x}{1-y}+\frac{x^2}{(1-x)(1-2y)}\,C\left(\frac{xy(1-y)}{(1-x)(1-2y)^2}\right),
\end{align}
where $C(z):=(1-\sqrt{1-4z})/(2z)$ is the generating function for
the Catalan numbers.
\end{thm}
This result has been rediscovered, in a different form, by
Br\"{a}nd\'en and Mansour~\cite[Theorem~4.7]{BrMa}. Our first
interesting result in this section is that the double sum
in~\eqref{eq:Burstein} can be easily transformed into a single sum.
This, which does not seem to have been noticed previously, leads to
the next result.
\begin{thm}\label{thm:SimplificationBurstein}
For any pattern $p\in\Sy_3$,  every $k\geq2$ and every $n\geq0$,
there holds
\begin{align}\label{eq:SimplificationBurstein}
|[k]^n(p)|=\frac{2^{n-2(k-2)}}{k-1}\sum_{j=0}^{k-2}(2k-2j-3)\binom{2j}{j}\binom{2(k-2-j)}{k-2-j}\binom{n+2j}{2j}.
\end{align}
\end{thm}

\begin{proof}
 The result is immediate by plugging the equation
\begin{align}\label{eq:simplificationSum1}
\sum_{i=j}^{k}
\frac{1}{i+1}\binom{2i}{i}\binom{2(k-i)}{k-i}=\frac{2k+1-2j}{k+1}\binom{2j}{j}\binom{2(k-j)}{k-j}\qquad(k\geq
j\geq 0)
\end{align}
into~\eqref{eq:Burstein}. To prove~\eqref{eq:simplificationSum1}, we
first shift the order of summation over $i$ by $j$. The left-hand
side of the resulting equation can be written using standard
hypergeometric notation as
\begin{align*}
\frac{1}{j+1}\binom{2j}{j}\binom{2(k-j)}{k-j}
{}_3F_2\left[{{1,j+\tfrac{1}{2},j-k}\atop
{j+2,j-k+\tfrac{1}{2}}};1\right].
\end{align*}
But, by Pfaff-Saalsch\"utz's ${}_3F_2$ identity, the ${}_3F_2$
series in the above expression simplifies to
\begin{align*}
{}_3F_2\left[{{1,j+\tfrac{1}{2},j-k}\atop
{j+2,j-k+\tfrac{1}{2}}};1\right]=\frac{(j+1)_{k-j}(\tfrac{3}{2})_{k-j}}{(k+1)_{k-j}(\tfrac{1}{2})_{k-j}}=\frac{(j+1)(2k-2j+1)}{k+1},
\end{align*}
where $(x)_n$ is for the Pochhammer symbol
$(x)_{n}=x(x+1)(x+2)\cdots(x+n-1)$. This ends the proof.
\end{proof}

\subsection{Explicit expression, generating function and recurrences}
We begin by noting that Proposition~\ref{prop:OPvsWORDS} can be
reformulated as follows. For a permutation $p$ consider the
generating functions
\begin{align*}
OP(x,y\,|\,p):=1+\sum_{n\geq k\geq 1}\op_{n,k}(p)\,x^k
y^n\quad\text{and}\quad W(x,y\,|\,p):=1+\sum_{n,k\geq
1}|[k]^n(p)|\,x^k y^n.
\end{align*}
Then  we have
 \begin{align}\label{eq:OPvsWORDS-genfunc}
 OP(x,y\,|\,p)&=\frac{1}{1+x}W(\frac{x}{1+x},y\,|\,p^{-1}).
 \end{align}
 To prove this, it suffices to equate coefficients of $x^{k} y^{n}$ on both
sides of~\eqref{eq:OPvsWORDS-genfunc} and then compare the resulting
equation with~\eqref{eq:OPvsWO}. We omit the details. Combining the
preceding equation with Theorem~\ref{thm:Burstein-GFandFor}, we
obtain the ordinary generating function of the array
$[\op_{n,k}(321)]_{n\geq k\geq 0}$.
\begin{thm}\label{thm:AvPartitions-GF}
We have
\begin{align}\label{eq:Burstein-GF}
\sum_{n\geq k\geq 0}\op_{n,k}(321)\,x^k
y^n=\frac{1}{1+x}+\frac{x}{(1+x)^2(1-y)}+\frac{x^2}{(1+x)^2(1-2y)}\,C\left(\frac{xy(1-y)}{(1-2y)^2}\right),
\end{align}
where $C(z)=(1-\sqrt{1-4z})/(2z)$.
\end{thm}
  This result was also obtained by Chen et al.~\cite{Chen} by a different approach.
Similarly, if we combine Proposition~\ref{prop:OPvsWORDS} with
Theorem~\ref{thm:SimplificationBurstein} and the obvious fact that
$|[1]^n(p)|=1$ for any permutation $p$ of length at least 2, we
arrive at the next formula.

\begin{thm}\label{thm:AvOP-length3}
For  $n\geq k\geq1$, there holds
\begin{align*}
\op_{n,k}(321)
&=(-1)^{k-1}k+\sum_{j=0}^{k-2}(-1)^{k-j}\binom{k}{j+2}2^{n-2j}
\sum_{i=0}^{j}\frac{2j-2i+1}{j+1}\binom{2i}{i}\binom{2(j-i)}{j-i}\binom{n+2i}{2i}.
\end{align*}
\end{thm}

Note that the second equation in~\eqref{eq:op321-k leq 2}
and~\eqref{eq:Godbole-k=3} are just the particular cases $k=2$ and
$k=3$ of the above formula.  The next three values read
\begin{align*}
\op_{n,4}(321)&= \frac{1}{3}(n^4+ 10 n^3-37 n^2-166 n+576)2^{n-6}-4,\\
\op_{n,5}(321)&=\frac{1}{9}(n^6+ 21 n^5-11 n^4-1125 n^3+1954 n^2+12984 n-36864)2^{n-10}+5,\\
\op_{n,6}(321)&=\frac{1}{45}(n^8+ 36 n^7+ 162 n^6-3528 n^5-8751 n^4+145044 n^3\\
            &\hspace{5cm}-144052 n^2-1463472 n+3686400)2^{n-14}-6.
\end{align*}
 More generally, if we interchange the sums over $i$ and $j$ in Theorem~\ref{thm:SimplificationBurstein},
 after some simplification, we obtain the following result.
\begin{cor}\label{cor:AvOP-length3}
For $k\geq2$ and $n\geq1$, there holds
\begin{align*}
\op_{n,k}(321)&=P_k(n)2^{n}+(-1)^{k-1}k,
\end{align*}
where, for fixed $k\geq2$, $P_k(n)$ is the polynomial in $n$ of
degree $2k-4$ given by
\begin{align*}
 P_k(n)&=\sum_{i=0}^{k-2}\frac{(n+1)_{2i}}{i!^2}\sum_{j=i}^{k-2}(-1)^{k-j}2^{-2j}\binom{k}{j+2}
\frac{2j-2i+1}{j+1}\binom{2(j-i)}{j-i}.
\end{align*}
\end{cor}

  An immediate but interesting consequence of the above result is
given in the next result. This claim follows directly by applying
the theory of rational generating functions. The proof details are
omitted.
\begin{cor}\label{cor:AvOP-rec}
For any fixed $k\geq2$, the following linear inhomogeneous
recurrence relation holds for $n\geq 2k-2$
\begin{align*}
\op_{n,k}(321)&=(-1)^{k}k-\sum_{j=1}^{2k-3}(-2)^{j}\binom{2k-3}{j}\op_{n-j,k}(321).
\end{align*}
\end{cor}
For example, when $k=3$, we get the recurrence relation
\begin{align*}
 \op_{n,3}(321)&=6 \,\op_{n-1,3}(321)-12\, \op_{n-2,3}(321)+8
\,\op_{n-3,3}(321)-3.
\end{align*}

A natural question to ask is whether formulas~\eqref{eq:op321-k=n}
and~\eqref{eq:Godbole-conjecture} for $\op_{n,n}(321)$ and
$\op_{n,n-1}(321)$ can be directly derived from
Theorem~\ref{thm:AvOP-length3}. The answer is yes but the path to
these formulas is far from obvious as illustrated in the appendix at
the end of this paper where we reproduce a calculation of
Krattenthaler~\cite{Kratt} for proving~\eqref{eq:Godbole-conjecture}
 directly  from Theorem~\ref{thm:AvOP-length3}. As we shall see in
 the next subsection,  it is more convenient to derive~\eqref{eq:op321-k=n}
 and~\eqref{eq:Godbole-conjecture}  from the next two recurrence relations. These can also be easily used to generate and extend
Table~\ref{table:1}.

\begin{cor}\label{cor:AvPartitions-Double rec}
For all integer $n$ (possibly negative) and $k\geq 0$, there holds
\begin{align}
\begin{split}\label{eq:DoubleRec1}
 (n+4)\op_{n+3,k+1}(321)&=(5n+14)\op_{n+2,k+1}(321)+(4n+10)\op_{n+2,k}(321)\\
                   &\quad  -(8n+14)\left(\op_{n+1,k+1}(321)+\op_{n+1,k}(321)\right)\\
                   &\quad  +(4n+4)\left(\op_{n,k+1}(321)+\op_{n,k}(321)\right),
\end{split}\\
\begin{split}\label{eq:DoubleRec2}
 (k+1)\op_{n+2,k+2}(321)&=(4k+4)\left(\op_{n+1,k+2}(321)-\op_{n,k+2}(321)\right)\\
                   &\quad -(k+2)\op_{n+2,k+1}(321)\\
                   &\quad +(8k+6)\left(\op_{n+1,k+1}(321)-\op_{n,k+1}(321)\right)\\
                   &\quad +(4k+2)\left(\op_{n+1,k}(321)-\op_{n,k}(321)\right).
\end{split}
\end{align}
\end{cor}

\begin{proof}
Let $A(x,y)$ denote the left-hand side of~\eqref{eq:Burstein-GF}.
Using Theorem~\ref{thm:AvPartitions-GF} (and a computer algebra
system), it is a routine matter to check that
\begin{multline*}
y\left(1-(5+4x)y+8(1+x)y^2-4(1+x)y^3\right)\frac{\partial}{\partial y}A(x,y)\\
+\left(1-2(2+x)y+6(1+x)y^2-4(1+x)y^3\right)A(x,y)
-\left(1-4y+6y^2-4y^3\right)=0,
\end{multline*}
\begin{multline*}
x\left((1-2y)^2+(1-8y+8y^2)x-4y(1-y)x^2\right)\frac{\partial}{\partial x}A(x,y)\\
-\left((1-2y)^2-(1-2y^2+2y)x+2y(1-y)x^2\right)A(x,y)
+\left((1-2y)^2-(1-2y^2+2y)x\right)=0.
\end{multline*}
Equating the coefficients of $x^{k+2} y^{n+3}$ on both sides of
these partial differential equations leads to the desired relations.
\end{proof}

 \subsection{Some properties of the diagonals of the array $[\op_{n,k}(321)]_{n\geq k\geq 1}$}
Here we are concerned with the sequences $(\op_{n+r,n}(321))_{n\geq
1}$, $r\geq0$. Our main goal is to establish Theorem~\ref{thm:Diag},
which can be seen as a far-reaching generalization
of~\eqref{eq:op321-k=n} and~\eqref{eq:Godbole-conjecture}. For ease
of notation, set $Y_r(n)=\op_{n+r,n}$ in the rest of this section.
By the convention~\eqref{eq:op321-IniVal},  $Y_r(n)$ is defined for
all $n\geq 0$ and all integers $r$.

Let $m\geq 0$ and $r$ be an integer.
Specializing~\eqref{eq:DoubleRec1} to $(n,k)=(m+r-2,m)$
and~~\eqref{eq:DoubleRec2} to $(n,k)=(m+r-1,m)$, after some
rearrangement, we obtain the equations
\begin{multline*}
(m+r+2)\,Y_{r}(m+1)-(4m+4r+2)\,Y_{r}(m)=(5m+5r+4)\,Y_{r-1}(m+1)\\
                     -(8m+8r-2)\left(Y_{r-2}(m+1)+Y_{r-1}(m)\right)+(4m+4r-4)\left(Y_{r-3}(m+1)+Y_{r-2}(m)\right),
\end{multline*}
\begin{multline*}
(m+2)\,Y_{r}(m+1)-(4m+2)\,Y_{r}(m)=(4m+4)\left(Y_{r-2}(m+2)-Y_{r-3}(m+2)\right)\\
                         +(8m+6)\left(Y_{r-1}(m+1)-Y_{r-2}(m+1)\right)-(4m+2)\,Y_{r-1}(m)-(m+1)\,Y_{r-1}(m+2).
\end{multline*}
Note that if we plug $r=0$ into the above equations, we obtain the
identity $$(m+2)\,Y_{0}(m+1)-(4m+2)\,Y_{0}(m)=0,$$ which yields
\begin{align}\label{eq:Y0}
 Y_{0}(m)=\frac{2m(2m-1)}{m(m+1)}\,Y_{0}(m-1)=\cdots=\frac{(2m)!}{m!(m+1)!}Y_{0}(0)=\frac{(2m)!}{m!(m+1)!}.
\end{align}
This is~\eqref{eq:op321-k=n}.

 If we now add $-(m+2)$ times the first equation to $(m+r+2)$ times
the second equation, after some simplification, we arrive at
\begin{align}\label{eq:rec-Yr}
\begin{split}
 6r\,Y_{r}(m)&= -(m+1)(m+r+2)\,Y_{r-1}(m+2)+ (3m^2+(3r+8)m-4r+4)\,Y_{r-1}(m+1)\\
            &\quad +(4m^2+4m(r+1)+14r-8)\,Y_{r-1}(m)+4(m+1)(m+r+2)\,Y_{r-2}(m+2)\\
            &\quad -(8m-10r+16)\,Y_{r-2}(m+1)-4(m+2)(m+r-1)\,Y_{r-2}(m)\\
            &\quad -4(m+1)(m+r+2)\,Y_{r-3}(m+2)-4(m+2)(m+r-1)\,Y_{r-3}(m+1).
\end{split}
\end{align}
 If we plug $r=1$ into the above equation, we obtain
\begin{align*}
6\,Y_{1}(m)&= -(m+1)(m+3)\,Y_{0}(m+2)+
(3m^2+11m)\,Y_{0}(m+1)+(4m^2+8m+6)\,Y_{0}(m),
\end{align*}
which simplifies, by~\eqref{eq:Y0}, to
\begin{align}\label{eq:Y1}
Y_{1}(m)&=\frac {(2m)!} {m!\,(m+2)!}(3m^2).
\end{align}
 This is exactly Godbole et al.'s formula~\eqref{eq:Godbole-conjecture}.
 Similarly, specializing~\eqref{eq:rec-Yr} to $r=2$, after a routine
computation, we obtain
\begin{align}\label{eq:Y2}
 Y_{2}(m)&=\frac{1}{2}\frac {(2m)!} {m!\,(m+3)!}(9m^4+16m^3+5m^2-6m).
\end{align}
More generally, we have the next result.
\begin{thm}\label{thm:Diag}
For each $r\geq 0$, there exists a polynomial $Q_r(n)$ in $n$ of
degree less than or equal to $2r$ such that
\begin{align*}
\op_{n+r,n}(321)&=\frac {(2n)!} {n!\,(n+r+1)!}\,Q_r(n)\qquad(n\geq
0).
\end{align*}

 Moreover, the polynomials $Q_r(n)$ satisfy for $r\geq 1$ the recurrence
\begin{align*}
6r\,Q_{r}(n)&= -2(2n+1)_3\,Q_{r-1}(n+2)+ 2(3n^2+(3r+8)n-4r+4)(2n+1)\,Q_{r-1}(n+1)\\
            &\quad +(4n^2+4n(r+1)+14r-8)(n+r+1)\,Q_{r-1}(n)+8(n+r+2)(2n+1)_3\,Q_{r-2}(n+2)\\
            &\quad -2(8n-10r+16)(n+r+1)(2n+1)\,Q_{r-2}(n+1)-4(n+2)(n+r-1)_3\,Q_{r-2}(n)\\
            &\quad -8(n+r+1)_2(2n+1)_3\,Q_{r-3}(n+2)-8(n+2)(n+r-1)_3(2n+1)\,Q_{r-3}(n+1),
\end{align*}
with initial conditions $Q_0(n)=1$ and $Q_j(n)=0$ for $j<0$.
\end{thm}
 Using the above recurrence, we easily obtain the first values of the polynomials $Q_{r}(n)$:
\begin{align*}
Q_{1}(n)&=3n^2,\\
Q_{2}(n)&=\tfrac{1}{2!}(9 n^4 + 16 n^3+5n^2-6n),\\
Q_{3}(n)&=\tfrac{1}{3!}(27n^6+144n^5+255n^4+114n^3-84n^2-96n),\\
Q_{4}(n)&=\tfrac{1}{4!}(81n^8+864n^7+3558n^6+6780n^5+5085n^4-1452n^3-4116n^2-2160n),\\
Q_{5}(n)&=\tfrac{1}{5!}(243n^{10}+4320n^9+31770n^8+123420n^7+265359n^6+284580n^5+50820n^4\\
        &\quad\quad-199200n^3-189792n^2-69120n).
\end{align*}
Note that the first two values are consistent with~\eqref{eq:Y1}
and~\eqref{eq:Y2}. We now turn to the proof of the theorem.

\begin{proof} Recall that we have set $Y_r(n)=\op_{n+r,n}(p)$ for $n\geq 0$ and $r\in\Z$. Let $Q_{r}(n)$ be
defined for $n\geq 0$ and $r\in\Z$ by
 \begin{align}\label{eq:defPol}
 Y_r(n)=\frac {(2n)!} {n!\,(n+r+1)!}\,Q_r(n),
 \end{align}
 where, by convention, we set $a!=1$ if $a$ is negative. In particular, by~\eqref{eq:op321-IniVal}, $Q_r(n)=Y_r(n)=0$ if $r< 0$.
 We also have $Q_{0}(n)=1$ by~\eqref{eq:Y0}.

 Suppose $r\geq 0$. To prove that the $Q_r(n)$'s defined by~\eqref{eq:defPol} satisfy the recurrence in Theorem~\ref{thm:Diag},
 it suffices to plug~\eqref{eq:defPol} into~\eqref{eq:rec-Yr} and then multiply the resulting equation by $n!(n+r+1)!/(2n)!$.
 After a routine simplification, we obtain the desired recurrence.

 We now prove that the $Q_r(n)$'s in~\eqref{eq:defPol}, $r\geq 0$, are polynomials in $n$ of degree at most~$2r$.
Note that, by~\eqref{eq:Y0}--\eqref{eq:Y2}, this is true for $0\leq
r\leq 2$. We proceed by induction on $r$.
 Suppose that $r$ is a fixed integer $\geq 3$ and that $Q_t(n)$ is a polynomial in $n$ of
degree at most $2t$ for $0\leq t\leq r-1$. Then, from the recurrence
in Theorem~\ref{thm:Diag}, we see that $Q_r(n)$ is a polynomial in
$n$ of degree at most $2r+1$. Furthermore, from the same recurrence,
the coefficient of $n^{2r+1}$ in $6r Q_r(n)$ is equal to the leading
coefficient of $Q_{r-1}(n)$ multiplied by the coefficient of $n^4$
in
\begin{align*}
-(2n+1)_4+(3n^2+(3r+8)n-4r+4)(2n+1)_2+(4n^2+4n(r+1)+14r-8)n(n+r+1).
\end{align*}
But the coefficient of $n^4$ in the above expression is easily shown
to be $-2^4+3\cdot 2^2+4$, so it is zero. This ends the proof.
\end{proof}


\medskip
{\bf Acknowledgements.} The author would like to thank Christian
Krattenthaler for spending time answering the question presented
after Corollary~\ref{cor:AvOP-rec} and for writing the solution
given below.



\section*{Appendix}
 This section contains a proof of~\eqref{eq:Godbole-conjecture}
 directly from the formula in Theorem~\ref{thm:AvOP-length3}.
 For convenience's reader, we restate the equation we are interested in:
\begin{align}\label{eq:AvOP-length3}
\op_{n+1,n}(321)&=3n^2\frac{(2n)!}{n!(n+2)!} \qquad(n\geq1).
\end{align}
By Theorem~\ref{thm:AvOP-length3}, the above equation is equivalent
to the relation
$$
\sum _{i=0} ^{n-2} \sum _{j=i} ^{n-2}(-1)^n\frac {2^
{n+1}\,(-1/4)^j} {(n+1)(j+1)} \frac {(n+1+2i)!\,(2j-2i+1)!}
{(j+2)!\,(n-2-j)!\,i!^2\,(j-i)!^2}=(-1)^n n
+3n^2\frac{(2n)!}{n!(n+2)!}.
$$

The proof of this sum evaluation that we reproduce here is due to
and was written by Christian Krattenthaler~\cite{Kratt1}. We
reproduce his calculation.

\begin{proof}
We want to evaluate the double sum
$$
\sum _{i=0} ^{n-2} \sum _{j=i} ^{n-2}(-1)^n\frac {2^
{n+1}\,(-1/4)^j} {(n+1)(j+1)} \frac {(n+1+2i)!\,(2j-2i+1)!}
{(j+2)!\,(n-2-j)!\,i!^2\,(j-i)!^2}.
$$
We start by rewriting the sum over $j$ in hypergeometric notation:
\begin{equation*} 
\sum _{i=0} ^{n-2} (-1)^{n+i}\frac {2^ {n-2i+1}} {(n+1)(i+1)} \frac
{(n+1+2i)!} {(i+2)!\,(n-2-i)!\,i!^2} {}_3 F_2\!\left[\begin{matrix}
i+1,\frac {3} {2},-n+i+2\\
i+2,i+3
\end{matrix};1\right].
\end{equation*}
To the $_3F_2$-series, we apply the transformation formula (see
\cite[(3.1.1)]{GaRa}),
\begin{equation*} 
{} _{3} F _{2} \!\left [ \begin{matrix} { a, b, -N}\\ { d,
e}\end{matrix} ;
   {\displaystyle 1}\right ]  =
  {\frac{ ({ \textstyle e-b}) _{N} }{({ \textstyle e}) _{N} }}\,
{} _{3} F _{2} \!\left [ \begin{matrix} { b, d-a, -N}\\ { d, 1 + b -
e
       - N}\end{matrix} ; {\displaystyle 1}\right ],
\end{equation*}
where $N$ is a nonnegative integer. Thereby the above expression is
transformed into
\begin{equation*} 
\sum _{i=0} ^{n-2} (-1)^{n+i}\frac {2^ {n-2i+1}} {(n+1)(i+1)} \frac
{(n+1+2i)!\,(i+\frac {3} {2})_{n-i-2}}
{(i+2)!\,(n-2-i)!\,i!^2\,(i+3)_{n-i-2}} {}_3
F_2\!\left[\begin{matrix}
\frac {3} {2},1,-n+i+2\\
i+2,\frac {3} {2}-n
\end{matrix};1\right].
\end{equation*}
Next we write the above $_3F_2$-series as a sum over $j$,
interchange the sums over $i$ and $j$, and finally write the (now)
inner sum over $i$ in hypergeometric notation. This produces the
expression
\begin{equation*} 
\sum _{j=0} ^{n-2} \frac {(-1)^{n}{2^ {n+1}} (\frac {3}
{2})_{n-2}\,(\frac {3} {2})_j\,(2-n)_j} {(n-2)!\,(j+1)!\,(\frac {3}
{2}-n)_j} {}_3 F_2\!\left[\begin{matrix}
\frac {n} {2}+1,\frac {n} {2}+\frac {3} {2},-n+j+2\\
\frac {3} {2},j+2
\end{matrix};1\right].
\end{equation*}
Here we apply the transformation formula (see \cite[Ex.~7, p.~98,
terminating form]{BailAA})
$$ 
{} _{3} F _{2} \!\left [ \begin{matrix} { a, b, -N}\\ { d,
e}\end{matrix} ;
   {\displaystyle 1}\right ]  =
{\frac{      ({ \textstyle d + e-a-b}) _{N} } {({ \textstyle e})
_{N} }} {} _{3} F _{2} \!\left [ \begin{matrix} {  d-a , d-b, -N }\\
{ d, d+e-a
      - b }\end{matrix} ; {\displaystyle 1}\right ]
$$
where $N$ is a non-negative integer. We obtain the expression
\begin{align} \notag
\sum _{j=0} ^{n-2}& \frac {(-1)^{j}{2^ {n+1}} (n-j-1)\, (\frac {3}
{2})_{j}\,(\frac {3} {2})_{n-j-2}} {(n-1)!} {}_3
F_2\!\left[\begin{matrix}
-\frac {n} {2},-\frac {n} {2}+\frac {1} {2},-n+j+2\\
\frac {3} {2},-n+j+1
\end{matrix};1\right]\\
\notag &\kern2cm = \sum _{j=0} ^{n-2} \sum _{i=0} ^{n-j-2} \frac
{(-1)^{j}{2^ {n-2i+1}}n (n-j-i-1)\, (\frac {3} {2}+i)_{j-i}\,(\frac
{3} {2})_{n-j-2}}
{i!\,(n-2i)!}\\
\label{eq:1} &\kern2cm = \sum _{i=0} ^{n-2} \sum _{j=0} ^{n-i-2}
\frac {(-1)^{j}{2^ {n-2i+1}}n (n-j-i-1)\, (\frac {3}
{2}+i)_{j-i}\,(\frac {3} {2})_{n-j-2}} {i!\,(n-2i)!}.
\end{align}
At this point, one should observe (by using the Gosper algorithm;
cf.\ \cite{GospAB} and \cite[\S~II.5]{PeWZAA} --- the particular
implementation that we used is the {\sl Mathematica} implementation
by Paule and Schorn \cite{PaScAA}) that
\begin{equation*} 
\frac {(-1)^{j}{2^ {n-2i+1}}n (n-j-i-1)\, (\frac {3}
{2}+i)_{j-i}\,(\frac {3} {2})_{n-j-2}}
{i!\,(n-2i)!}=G(n,i;j+1)-G(n,i;j),
\end{equation*}
where
\begin{multline*} 
G(n,i;j)=\frac {(-1)^{j+1}{2^ {n-2i}}n \, (\frac {3}
{2}+i)_{j-i}\,(\frac {3} {2})_{n-j-1}}
{(n+1)(n+2)\,i!\,(n-2i)!} 
\left(2 n^2+2
   n-2 i
   n-4 i-2 j n-2 j-1\right).
\end{multline*}
Consequently, the sum over $j$ in \eqref{eq:1} is a telescoping sum,
and thus the right-hand side of \eqref{eq:1} simplifies to
\begin{multline*} 
\sum _{i=0} ^{n-2} \frac {(-1)^{n-i}{2^ {n-2i+1}}n \, (\frac {3}
{2})_{n-i}}
{(n+1)(n+2)\,i!\,(n-2i)!}
+ \sum _{i=0} ^{n-2} \frac {{2^ {n-2i}}n \, (\frac {3} {2})_{n-1}}
{(n+1)(n+2)\,i!\,(n-2i)!\,(\frac {3} {2})_i} \left(2 n^2+2
   n-2 i
   n-4 i-1\right).
\end{multline*}
Writing the sums in hypergeometric notation, we obtain
\begin{multline*} 
\frac {(-1)^{n}{2^ {n+1}}n \, (\frac {3} {2})_{n}} {(n+2)!} {}_2
F_1\!\left[\begin{matrix}
-\frac {n} {2},-\frac {n} {2}+\frac {1} {2}\\
-n-\frac {1} {2}
\end{matrix};1\right]
+(2n^2+2n-1) \frac {{2^ {n}}n \, (\frac {3} {2})_{n-1}} {(n+2)!}
{}_2 F_1\!\left[\begin{matrix}
-\frac {n} {2},-\frac {n} {2}+\frac {1} {2}\\
\frac {3} {2}
\end{matrix};1\right]\\
- \frac {{2^ {n-1}}n \, (\frac {5} {2})_{n-2}} {(n+1)\,(n-2)!} {}_2
F_1\!\left[\begin{matrix}
-\frac {n} {2}+1,-\frac {n} {2}+\frac {3} {2}\\
\frac {5} {2}
\end{matrix};1\right].
\end{multline*}
The $_2F_1$-series can be evaluated by means of the Chu--Vandermonde
summation formula (cf.\ \cite[(1.7.7); Appendix (III.4)]{SlatAC})
$$
{} _{2} F _{1} \!\left [ \begin{matrix} { a, -N}\\ { c}\end{matrix}
; {\displaystyle
   1}\right ]  = {\frac {({ \textstyle c-a}) _{N} }
    {({ \textstyle c}) _{N} }},
$$
where $N$ is a nonnegative integer. After little manipulation, we
arrive at the expression
\begin{multline*} 
(-1)^nn +\frac {(2n^2+2n-1)\,(2n)\,(2n)!} {(n+1)!\,(n+2)!}
-\frac {(2n)!} {(n+1)\,(n-2)!\,(n+1)!}
= (-1)^nn +\frac {3n\,(2n)!} {(n-1)!\,(n+2)!},
\end{multline*}
as required.
\end{proof}


\end{document}